\documentclass[dvipdfmx]{amsart}
\usepackage[english]{babel}
\usepackage[utf8]{inputenc}
\usepackage{hyperref}
\usepackage{mathtools}
\usepackage{amsfonts}
\usepackage{amsthm}
\usepackage{amssymb}
\usepackage{bbm}
\usepackage{fullpage}
\usepackage{stmaryrd}
\usepackage{graphicx,color}
\usepackage{url}

\newcommand{\floor}[1]{\lfloor #1 \rfloor}

\numberwithin{equation}{section}

\newtheorem{theorem}{Theorem}[section]

\newtheorem{lemma}[theorem]{Lemma}

\newtheorem{thmx}{Theorem}

\theoremstyle{remark}
\newtheorem{remark}[theorem]{Remark}

\theoremstyle{definition}

\newcommand{\Z}{\mathbb Z}
\newcommand{\R}{\mathbb R}
\newcommand{\N}{\mathbb N}

\newcommand{\1}{\mathbbm 1}
\newcommand{\eps}{\varepsilon}
\newcommand{\dd}{\text{\upshape{d}}}
\renewcommand{\P}{\mathbb P}
\newcommand{\E}{\mathbb E}

\newcommand{\p}{\mathfrak{p}}
\renewcommand{\a}{\mathfrak{a}}

\newcommand{\kkappa}{{\boldsymbol \kappa}}

\newcommand{\ddelta}{{\boldsymbol \delta}}

\title{On large deviation rate functions for a continuous-time directed polymer in weak disorder}

\author{Ryoki Fukushima}
\address{Institute of Mathematics, University of Tsukuba, 1-1-1 Tennodai, Tsukuba, Ibaraki 305--8571, Japan}
\email{ryoki@math.tsukuba.ac.jp}

\author{Stefan Junk}
\address{Research Institute for Mathematical Sciences, Kyoto University, Kyoto 606--8502, Japan}
\curraddr{Institute of Mathematics, University of Tsukuba, 1-1-1 Tennodai, Tsukuba, Ibaraki 305--8571, Japan}
\email{sjunk@math.tsukuba.ac.jp}

\begin{document}

\begin{abstract}
We consider the endpoint large deviation for a continuous-time directed polymer in a L\'evy-type random environment. When the space dimension is at least three, it is known that the so-called weak disorder phase exists, where the quenched and annealed free energies coincide. We prove that the rate function agrees with that of the underlying random walk near the origin in the whole interior of the weak disorder phase. 
\end{abstract}

\maketitle

\section{Introduction}
\label{sec:intro}
In this article, we study a continuous-time directed polymer model in a L\'evy-type random environment. This type of model is known to exhibit a phase transition when space dimension is greater than or equal to three. More precisely, the polymer behaves like simple random walk when the disorder is weak, whereas it tends to localize when the disorder is strong; see a recent survey~\cite{comets_st_flour} for more detail. While most of the research has been devoted to typical behaviors of the polymer, it is also shown in~\cite[Theorem~6.1]{RAS14} and~\cite[Exercise 9.1]{comets_st_flour} that \emph{deeply} inside the weak disorder phase, even the rate function of the large deviation principle for the polymer endpoint coincides with that for the simple random walk near the origin. The aim of this paper is to present a simple argument that extends this result to the \emph{whole} interior of weak disorder phase for continuous-time directed polymer model in a L\'evy-type random environment. 

Let $(\omega=(\omega_x)_{x\in\Z^d},\P)$ be an independent and identically distributed (i.i.d.)~collection of real-valued L\'evy processes with characteristic triple $(0,\sigma^2,\rho)$. We assume that the L\'evy measure $\rho$ has finite mass, is supported on $[-1,\infty)$ and satisfies
\begin{align*}
\int_{[-1,\infty)}r\rho(\dd r)<\infty.
\end{align*}
The assumption that $\rho$ has finite mass is not essential, but leads to a simpler presentation. In particular, it implies that the L\'evy processes are of pure jump type and we can therefore write
\begin{align}
\omega_x(t)=\sigma^2B_x(t)+\int_{[0,t]\times[-1,\infty)}r N_x(\dd s\dd r),
\label{eq:noise}
\end{align}
where $(B_x)_{x\in\Z^d}$ is an i.i.d.~collection of Brownian motions and $(N_x)_{x\in\Z^d}$ is an i.i.d.~collection of Poisson point processes on $\R_+\times[-1,\infty)$, which is independent of $(B_x)_{x\in\Z^d}$, with intensity measure $\dd s\,\rho(\dd r)$. Given $\omega$, we define a new process $(L_x)_{x\in\Z^d}$ by
\begin{equation}
 L_x(t)=\sigma^2B_x(t)+\int_{[0,t]\times[-1,\infty)}\log(1+r) N_x(\dd s\dd r)
\label{eq:potential}
\end{equation}
with the convention $\log 0=-\infty$, and $-\infty$ being an absorbing state for this process. For a path $x\colon\R_+\to\Z^d$ that is right-continuous and has left limits everywhere (c\'adl\'ag), we define the Hamiltonian by
\begin{align*}
H_t(\omega,x)=\int_{[0,t]}\dd L_{x(s)}(s)=\sum_{y\in\Z^d}\int_{[0,t]}\1_{\{x(s)=y\}}\,\dd L_{y}(s).
\end{align*}
Let $(X=(X_t)_{t\geq 0},P^\kappa)$ denote the random walk which starts at the origin and jumps to a site chosen uniformly from the nearest neighbor sites at rate $\kappa>0$. The polymer measure of $P^\kappa$ is the random probability measure $\mu_{\omega,t}^\kappa$ defined by
\begin{align*}
\mu_{\omega,t}^\kappa(\dd X)= \frac{1}{Z_{\omega,t}^\kappa}e^{H_t(\omega,X)} P^\kappa(\dd X),
\end{align*}
with the convention $e^{-\infty}=0$, where the normalizing constant is given by $Z_{\omega,t}^\kappa= E^\kappa[e^{H_t(\omega,X)}]$. This probability measure gives more weight to paths along which the environment is increasing, and discourages paths that observe a decreasing environment. In particular, we interpret the set $\{(t,x)\colon\omega_x(t)=\omega_x(t-)-1\}$ as hard obstacles in space-time and note that the polymer is conditioned to avoid this set.

One may wonder why we introduce the process $\omega$ in~\eqref{eq:noise} first and transform it to $L$ in~\eqref{eq:potential}. This is mostly due to historical reasons. We regard $Z_{\omega,t}$ as the ``partition function'' of Gibbs measure $\mu_{\omega,t}^\kappa$, but in most of the earlier studies~\cite{ACM92,Nob97,shiga,shiga2,shiga3,mountford,thesis}, it is regarded as the solution of a stochastic partial differential equation called the parabolic Anderson model. More precisely, the point-to-point partition function
\begin{equation}
Z_{\omega,t,x}^\kappa\coloneqq E^\kappa\big[e^{H_t(\omega,X)}\1_{\{X(t)=x\}}\big] 
\label{eq:P2P}
\end{equation}
is equal to $u(t,x)$, where $u$ denotes the solution to the following initial value problem for a stochastic heat equation with L\'evy noise,
\begin{equation}
\begin{split}
\tfrac{\partial}{\partial s} u(s,x)&=\tfrac{\kappa}{2d} (\Delta u)(s,x) + u(s,x) \dd\omega_x(s), \\
u(0,\cdot)&=\1_{\{0\}}(\cdot),
\end{split}
\label{eq:PAM}
\end{equation}
where $\Delta$ denotes the discrete Laplacian. 
\begin{remark}
In the literature on random media, it is customary to distinguish the quenched (fixed media) and annealed (averaged media) models. For instance, the above $Z_{\omega,t}^\kappa$ is the quenched partition function whereas the annealed partition function is $\E[Z_{\omega,t}^\kappa]$. However, the annealed model in our setting turns out to be too simple to be interesting. For this reason, we concentrate on the quenched model and do not make this distinction throughout this article. 
\end{remark}

Before reviewing some known results, we introduce more general notation that will be useful later. We use bold symbols to highlight multi-dimensional parameters. Let $\kkappa=(\kappa_e)_{|e|_1=1}\in(0,\infty)^{2d}$, and write $P^{\kkappa}$ for the law of the random walk with generator
\begin{align*}
(L^{\kkappa}f)(x)=\sum_{|e|_1=1} \frac{\kappa_e}{2d} \big(f(x+e)-f(x)\big).
\end{align*}
We write $Z^{\kkappa}_{\omega,t}$ (resp. $Z^{\kkappa}_{\omega,t,x}$) for the corresponding partition function (resp. point-to-point partition function). Note that $P^\kappa=P^{\kappa\mathbf{1}}$, where $\mathbf 1=(1,\dots,1)$. Let us first list the known properties of the partition functions:
\begin{thmx}\label{thm:known}
\begin{enumerate}
\item[(i)] For every $\kkappa\in(0,\infty)^{2d}$ and $x\in\R^d$, there exist $\p(\kkappa),\p(\kkappa,x)\in\R$ such that almost surely,
\begin{alignat}{2}
\lim_{t\to\infty}\frac 1t\log Z_{\omega,t}^\kkappa&=\lim_{t\to\infty}\frac 1t\E\big[\log Z_{\omega,t}^\kkappa\big]&&=\p(\kkappa),\\
\lim_{t\to\infty,t\in\N}\frac 1t\log Z_{\omega,t,[tx]}^\kkappa&=\lim_{t\to\infty}\frac 1t\E\big[\log Z_{\omega,t,[tx]}^\kkappa\big]&&=\p(\kkappa,x).\label{eq:p2p}
\end{alignat}
Moreover, the function $\kappa\mapsto \p(\kappa\mathbf1)$ is continuous.
\item[(ii)]Let $\a\coloneqq\frac{\sigma^2}2+\int_{[-1,\infty)}r\rho(\dd r)$. The process $(W_t^\kkappa)_{t\geq0}$ defined by $W_t^\kkappa\coloneqq Z_t^\kkappa e^{-\a t}$ is a non-negative martingale, and its almost sure limit $W_\infty^\kkappa\coloneqq\lim_{t\to\infty}W_t^\kkappa$ satisfies a zero-one law,
\begin{align*}
\P(W_\infty^\kkappa=0)\in\{0,1\}.
\end{align*}
The two cases are referred to as \textbf{strong disorder} ($W_\infty^\kkappa=0$) and \textbf{weak disorder} ($W_\infty^\kkappa>0$).
\item[(iii)]There exist critical values $0<\overline\kappa_{\rm cr}(d)\leq\kappa_{\rm cr}(d)\leq\kappa_{\rm cr}^{L^2}(d)\leq \infty$ such that
\begin{itemize}
 \item $\p(\kappa\mathbf1)<\a$ for $\kappa<\overline{\kappa}_{\rm cr}$, and $\p(\kappa\mathbf1)=\a$ for $\kappa\geq \overline\kappa_{\rm cr}$.
 \item $W_\infty^{\kappa\mathbf1}=0$ for $\kappa<\kappa_{\rm cr}$, and $W_\infty^{\kappa\mathbf1}>0$ for $\kappa>\kappa_{\rm cr}$.
 \item $(W_t^{\kappa\mathbf1})_{t\geq 0}$ is $L^2$-bounded if and only if $\kappa>\kappa_{\rm cr}^{L^2}$.
\end{itemize}
 \item[(iv)] In dimension $d\geq 3$, all critical values are finite and $\kappa_{\rm cr}(d)<\kappa_{\rm cr}^{L^2}(d)$.
\end{enumerate}
\end{thmx}
\begin{remark}
\begin{enumerate}
 \item The functions $\p(\kkappa)$ and $\p(\kkappa,x)$ are called the free energy and point-to-point free energy, respectively. From (ii), it in particular follows that $\p(\kkappa)\le \a$ for any $\kkappa\in(0,\infty)^{2d}$. This is called the annealed bound (cf.~\cite[(2.1.3)]{comets_st_flour}).  
 \item When $d\le 2$, we believe $\overline\kappa_{\rm cr}(d)=\kappa_{cr}=\infty$ since the corresponding result is proved in~\cite{LacoinVSD} for the discrete-time model.
\end{enumerate}
\end{remark}

In this article, we study the large deviation principle (LDP) for the endpoint distribution under $\mu_{\omega,t}^\kappa$. The following abstract existence result holds:

\begin{thmx}\label{thm:existence}
For every $\kappa>0$ and every $d\ge 1$, the sequence $(\mu_{\omega,t}^{\kappa}(X_t/t\in\cdot))_{t> 0}$ satisfies an LDP with deterministic, good, convex rate function $J^\kappa(x)\coloneqq\p({\kappa\mathbf1})-\p({\kappa\mathbf1},x)$, $\P$-almost surely.
\end{thmx}

Theorems \ref{thm:known} and \ref{thm:existence} are well-known, at least for the related discrete-time polymer model. In the appendix, we briefly outline how to prove them in our continuous-time setting by providing some references. Our main result compares $J^\kappa$ for $\kappa\ge \overline\kappa_{\rm cr}$ with the large deviation rate function $I^\kappa$ of $(P^\kappa(X_t/t\in\cdot))_{t> 0}$ which has the following explicit form:
\begin{align}\label{eq:rate_fct}
I^\kappa(x_1,...,x_d)=\sum_{i=1}^d \left\{x_i\sinh^{-1}\left(\frac{d x_i}{\kappa}\right)-\sqrt{x_i^2+\frac{\kappa^2}{d^2}}+\frac\kappa d\right\}.
\end{align}

\begin{theorem}\label{thm:main}
Let $d\geq 3$. 
Then the following hold:
\begin{enumerate}
\item[(i)] If $\kappa>\overline\kappa_{\rm cr}$, then $J^\kappa$ and $I^\kappa$ coincide in a neighborhood of the origin.  
\item[(ii)] If $\kappa\geq\overline\kappa_{\rm cr}$, then $J^\kappa(x) \ge I^\kappa(x)$ for all $x$.
\end{enumerate}
\end{theorem}

The same conclusion has been proved in~\cite[Theorem~6.1]{RAS14} and~\cite[Exercise 9.1]{comets_st_flour} for the discrete-time directed polymer model under the stronger assumption of $L^2$-boundedness, which would correspond to $\kappa>\kappa_{\rm cr}^{L^2}$ in our notation. The continuous-time model has the advantage that the parameter of the model changes from the ``inverse temperature'' $\beta$ to the jump rate $\kappa$, so that we can use a convolution property of the continuous-time random walk, which allows us to compare the partition functions for different jump rates; see Theorem~\ref{thm:comp} below. 

\begin{remark}
The coincidence and difference of the quenched and annealed rate functions are studied also in the setting of random walk in random environment:~\cite{Yil09,YZ10,Yil11,RASY17b,BMRS19}. 
\end{remark}

It is an important open problem (see, e.g.,~\cite[Open Problem 9.3]{comets_st_flour}) to prove that the rate function for the directed polymer model is strictly convex near the origin. One of the major reasons is that it is a key to prove the so-called scaling relation, as is proved in~\cite{AD13}. Although Theorem~\ref{thm:main} gives an affirmative answer, we think it is of limited interest in this aspect (except possibly for $\kappa=\overline\kappa_{\rm cr}$) since the scaling exponents are known in the weak disorder phase. Apart from the results in weak disorder, the strict convexity is known only for (i) certain exactly solvable models~\cite{seppalainen} for which the rate function is explicitly known, and (ii) the Brownian polymer model in continuous space~\cite{cometsyoshida} for which the rate function agrees with that of the Brownian motion in both strong and weak disorder. The last result is due to a special translation invariance property of the Brownian bridge and the Poisson point process. 

For the simple random walk model, the bridge loses entropy as the endpoint moves away from the origin, and we expect that the polymer rate function is strictly larger than that of simple random walk in strong disorder. Moreover, it is conjectured that strong disorder holds at the critical value $\kappa=\overline\kappa_{\rm cr}$, so we expect that the conclusion of Theorem~\ref{thm:main}~(i) does not extend to $\kappa\leq \overline\kappa_{\rm cr}$.

\section{Proof of the main result}

The proof of Theorem \ref{thm:main} relies on the comparison result from \cite{self_comp}. If $(X=(X_t)_{t\geq 0}, Q)$ and $(X'=(X'_t)_{t\geq 0}, Q')$ are two independent processes on the space of c\'adl\'ag paths on $\Z^d$, we write $Q*Q'$ for the law of $(X_t+X_t')_{t\geq 0}$. We write $P\preceq_*Q$ if there exists a probability measure $Q'$ such that $P=Q*Q'$. Note that
\begin{align*}
P^{\kkappa_1+\kkappa_2}=P^{\kkappa_1}*P^{\kkappa_2}
\end{align*}
for all $\kkappa_1,\kkappa_2\in(0,\infty)^{2d}$, and therefore both $P^{\kkappa_1+\kkappa_2}\preceq_* P^{\kkappa_1}$ and $P^{\kkappa_1+\kkappa_2}\preceq_* P^{\kkappa_2}$ hold.

\begin{thmx}[{\cite[Theorem 1]{self_comp}}]\label{thm:comp}
Let $P$ and $Q$ be two probability measures on c\'adl\'ag paths on $\Z^d$, and write 
\begin{align*}
Z^P_{\omega,t}\coloneqq E_P[e^{H_t(\omega,X)}]\quad\text{ and }\quad Z^Q_{\omega,t}\coloneqq E_Q[e^{H_t(\omega,X)}]
\end{align*}
for the associated partition functions. Then $P\preceq_*Q$ implies that for any $f\colon[0,\infty)\to[-\infty,\infty)$ concave,
\begin{align*}
\E\big[f(Z^Q_t)\big]\leq\E\big[f(Z^P_t)\big].
\end{align*}
\end{thmx}

The maximal elements with respect to $\preceq_*$ are the Dirac measures on a deterministic path. Intuitively, the partition function is smaller (in the concave stochastic order) if there is ``less randomness'' in the underlying random walk, in the sense that its law is large with respect to $\preceq_*$. 
We sketch the proof of Theorem \ref{thm:comp} for the readers' convenience. Following the above notation and using Jensen's inequality, we have 
\begin{align*}
\E\big[f(Z^P_t)\big]&=\E\big[f(E_Q\otimes E_{Q'}[e^{H_t(\omega,X+X')}])\big]\\
&\ge E_{Q'}\big[\E\big[f(E_Q[e^{H_t(\omega_{X'},X)}])\big]\big],
\end{align*}
where $\omega_{X'}$ is the environment seen from $X'$. But for any fixed $X'$, the law of $\omega_{X'}$ is the same as that of $\omega$ and hence we can get rid of $E_{Q'}$ and $X'$ from the last line. 

Next, we compute the Cram\'er transform of $P^\kappa$:
\begin{lemma}\label{lem:generator}
Let $\lambda\in\R^d$ and define a probability measure $Q$ by
\begin{align*}
Q(\dd X_t)=\frac{e^{\langle\lambda, X_t\rangle}}{E^\kappa[e^{\langle\lambda, X_t\rangle}]}P^\kappa(\dd X_t).
\end{align*}
Then $((X_s)_{s\in[0,t]},Q)$ has law $P^{\kkappa(\lambda)}$, where $\kkappa(\lambda)=\{\kappa e^{\langle\lambda, e\rangle}\}_{|e|=1}\in(0,\infty)^{2d}$. 
\end{lemma}

\begin{proof}
It is easy to see that $((X_s)_{s\in[0,t]},Q)$ is a Markov process. Thus it suffices to identify its generator:
\begin{align*}
(L^Qf)(x)&=\lim_{t\downarrow 0}\frac{1}{t}\left(\frac{E_x^\kappa[f(X_t)e^{\langle\lambda,X_t\rangle}]}{E_x^\kappa[e^{\langle \lambda, X_t\rangle}]}-f(x)\right)\\
&=e^{-\langle \lambda,x\rangle}\big((L^{\kappa}fe^{\langle \lambda,\cdot\rangle})(x)- f(x)(L^{\kappa}e^{\langle \lambda,\cdot\rangle})(x)\big)\\
&=\frac{\kappa}{2d} \sum_{|e|=1} \Big(e^{\langle \lambda,e \rangle}f(x+e)-f(x)-f(x)(e^{\langle\lambda,e\rangle}-1)\Big)\\
&=(L^{\kkappa(\lambda)}f)(x).\qedhere
\end{align*}
\end{proof}

We now prove the main result:
\begin{proof}[Proof of Theorem \ref{thm:main}]
Let $\lambda\in\R^d$. By Lemma \ref{lem:generator}, we have
\begin{align*}
\mu_{\omega,t}^{\kappa}[e^{\langle \lambda, X_t \rangle}]=\frac{Z^{\kkappa(\lambda)}_{\omega,t}}{Z^{\kappa}_{\omega,t}} e^{t \Lambda^\kappa(\lambda)},
\end{align*}
where $\Lambda^\kappa(\lambda)\coloneqq \log E^\kappa[e^{\langle \lambda,X_1 \rangle}]$. Theorem \ref{thm:known}~(i) shows that almost surely
\begin{align*}
{\widehat\Lambda}^\kappa(\lambda)\coloneqq\lim_{t\to\infty}\frac 1t\log \mu_{\omega,t}^{\kappa}[e^{\langle\lambda,X_t\rangle}]=\p(\kkappa(\lambda))-\p(\kappa\mathbf1)+\Lambda^\kappa(\lambda).
\end{align*}
We will show that for all $\lambda$ small enough,
\begin{align}\label{eq:todo}
\widehat\Lambda^\kappa(\lambda)=\Lambda^\kappa(\lambda).
\end{align}
Once we have this identification of the cumulant generating function, we can apply the G\"{a}rtner--Ellis theorem~\cite[Theorem 2.3.6 (c)]{DZ} to conclude that $(\mu_{\omega,t}^\kappa(X_t/t\in\cdot))_{t\in\N}$ satisfies an LDP near the origin with the rate function $I^\kappa$. 
In order to prove \eqref{eq:todo}, we introduce
\begin{align}
&\underline{\kappa}(\lambda)\coloneqq \kappa\min_{|e|=1}e^{\langle \lambda,e\rangle},\quad
\ddelta_1(\lambda)\coloneqq \kkappa(\lambda)-\underline{\kappa}(\lambda)\mathbf{1},\\
&\overline{\kappa}(\lambda)\coloneqq \kappa\max_{|e|=1}e^{\langle \lambda,e\rangle},\quad
{\ddelta}_2(\lambda)\coloneqq \overline{\kappa}(\lambda)\mathbf{1}-\kkappa(\lambda),
\end{align}
so that
\begin{align*}
&P^{\kkappa(\lambda)}=P^{\underline{\kappa}(\lambda)\mathbf{1}}*P^{\ddelta_1(\lambda)},\\
&P^{\kkappa(\lambda)}*P^{\ddelta_2(\lambda)}=P^{\overline{\kappa}(\lambda)\mathbf{1}}.
\end{align*}
Applying Theorem \ref{thm:comp} with $f(x)=\log x$, we get
\begin{align*}
\E\big[\log Z^{\underline{\kappa}(\lambda)}_{\omega,t}\big]
\leq \E\big[\log Z^{\kkappa(\lambda)}_{\omega,t}\big]
\le\E\big[\log Z^{\overline{\kappa}(\lambda)}_{\omega,t}\big],
\end{align*}
and Theorem \ref{thm:known}~(i) implies that 
\begin{align*}
\p(\underline\kappa(\lambda)\mathbf1)\leq \p(\kkappa(\lambda))\leq \p(\overline\kappa(\lambda)\mathbf1).
\end{align*}
Since $\kappa>\overline\kappa_{\rm cr}$ and $\lim_{|\lambda|\to 0}\underline{\kappa}(\lambda)=\lim_{|\lambda|\to 0}\overline{\kappa}(\lambda)=\kappa$, we have for all sufficiently small $\lambda$, 
\begin{align*}
\p(\underline\kappa(\lambda)\mathbf1)=\p(\kkappa(\lambda))=\p(\overline\kappa(\lambda)\mathbf1)=\p(\kappa\mathbf1)=\a
\end{align*}
and~\eqref{eq:todo} follows.

Finally, for part (ii), we have $\p(\overline\kappa_{\rm cr}\mathbf1)=\a$ by Theorem \ref{thm:known}~(ii), while $\p(\kkappa)\le \a$ holds for any $\kkappa\in(0,\infty)^{2d}$ by the annealed bound. Thus, instead of~\eqref{eq:todo}, we have for any $\lambda\in\R^d$ and $\kappa\geq\overline\kappa_{\rm cr}$, 
\begin{align*}
\widehat\Lambda^{\kappa}(\lambda) \le \Lambda^{\kappa}(\lambda)
\end{align*}
and Theorem~\ref{thm:main}~(ii) follows from this. 
\end{proof}

\begin{remark}
This argument does not extend to the regime $\kappa<\overline\kappa_{\rm cr}$.
Suppose we want to prove that the graph of $J^\kappa$ around zero has a curvature bonded away from zero, which is weaker than Theorem~\ref{thm:main}~(ii). For this purpose, it suffices to prove instead of~\eqref{eq:todo} that
\begin{equation}\label{eq:new_todo}
\p(\kkappa(\lambda)) \le \p(\kappa\mathbf1)-c|\lambda|^2 
\end{equation}
for some $c>0$ and all sufficiently small $\lambda$. Our method above is to replace $\p(\kkappa(\lambda))$ by $\p(\underline\kappa(\lambda)\mathbf1)$, but in strong disorder this should introduce an error-term that is much larger than the quadratic term. More precisely, we expect that $\kappa\mapsto \p(\kappa\mathbf1)$ is strictly increasing for $\kappa\in(0,\overline\kappa_{\rm cr})$, so that $\p(\underline{\kappa}(\lambda)\mathbf1) \approx \p(\kappa\mathbf1)-c' |\ddelta_2(\lambda)|$ for $\lambda\to 0$. However, $|\ddelta_2(\lambda)|$ decays only linearly as $\lambda\to 0$.
\end{remark}

\section*{Appendix: Known results}
Results similar to Theorem \ref{thm:known} are well-known for directed polymers in discrete time, see \cite{comets_st_flour}. If we assume that there are no hard obstacles, i.e. $\rho([-1,-1+\eps])=0$ for some $\eps>0$, all claims can be obtained using the same arguments as in discrete time.

\begin{proof}[Proof of Theorem \ref{thm:known}] The existence of $\p(\kappa\mathbf1)$ as an $L^1$-limit has been shown in \cite[Theorem 1.2]{shiga} and \cite[Theorem 3.1]{shiga2}, while the existence of the almost sure limit is shown in \cite[Theorem 1.1]{shiga3} (under the assumption $\rho(\{-1\})=0$). The existence of the almost sure limit in the hard obstacle case, $\sigma^2=0$ and $\rho=\delta_{\{-1\}}$, is presumably well-known, as it is used in \cite{mountford}. A proof for the existence of $\p(\kappa\mathbf 1)$ and $\p(\kappa\mathbf1,x)$ as almost sure limits in the hard obstacle case, as well as the continuity of $\kappa\mapsto \p(\kappa\mathbf 1)$, can be found in \cite[Propositions 5.5 and 5.6]{thesis}. Those arguments also apply to general environments and with $\kappa\mathbf1$ replaced by a general $\kkappa\in (0,\infty)^{2d}$. 

That $W_t^\kkappa$ is a martingale follows by a standard argument (cf.~\cite[Section 3.1]{comets_st_flour}) once we show $\E[Z_{t}^\kkappa]=e^{\a t}$. To this end, note that since $(L_x)_{x\in\Z^d}$ in~\eqref{eq:potential} is an independent family of L\'evy processes, for any fixed c\'adl\'ag path on $\Z^d$ that jumps at times $0=t_0<t_1<t_2<\cdots< t_n<t$, the Hamiltonian
\begin{align*}
H_t(\omega,x)=\sum_{k=0}^n(L_{x(t_{k-1})}(t_k)-L_{x(t_{k-1})}(t_{k-1}))+(L_{x(t_n)}(t)-L_{x(t_n)}(t_n))
\end{align*}
has the same law as $L_0(t)$. This together with Fubini's theorem and the fact that $\a$ is the Laplace exponent for $L_0(t)$ imply 
\begin{align*}
\E[Z_{\omega,t}^\kkappa]&=E^\kappa\big[\E\big[e^{H_t(\omega,X)}\big]\big]\\ 
&=\E\left[\exp\{L_0(t)\}\right]\\
&=e^{\a t}.
\end{align*}
as desired. The zero-one law from part (ii) has been shown in \cite[Theorem 1.1]{shiga} in the case $\sigma^2=0$, $\rho=\delta_{\{-1\}}$, and the same argument also applies to general environments.

The existence of $\overline\kappa_{\rm cr}$ is shown in \cite[Theorem 1.3]{shiga} for $\sigma^2=0$, $\rho=\delta_{\{-1\}}$, while for general environments it follows from the fact that $\kappa\mapsto \p(\kappa\mathbf1)$ is increasing, which was shown in \cite[(4.15)]{self_comp}. To show the existence of $\kappa_{\rm cr}$, one first follows the arguments from \cite[Proposition 3.1]{yoshida_diffusive} to show that $W^{\kappa\mathbf1}_\infty>0$ is equivalent to $L^1$-convergence of $(W^{\kappa\mathbf1}_t)_{t\geq 0}$, and then applies the same argument as in \cite[(4.8)]{self_comp}. The relation $\overline\kappa_{\rm cr}\leq\kappa_{\rm cr}\leq\kappa_{\rm cr}^{L^2}$ is obvious. 
 
Finally, in dimension $d\geq 3$, the same argument as in the discrete-time case shows $\kappa_{\rm cr}^{L^2}(d)<\infty$, see \cite[Theorem 3.3]{comets_st_flour}. The second statement in part (iv) has been shown for the discrete-time model in \cite{birkner}, and the generalization to continuous time follows along the same lines, using \cite[Theorem 7]{birkner_thesis}. 
\end{proof}
 
The LDP for a discrete-time polymer model is established in~\cite{carmona_hu}, but to the best of our knowledge, it is not available for continuous time random walk models in the literature. 

\begin{proof}[Proof of Theorem \ref{thm:existence}]
We start by considering integer times, i.e., the sequence $(\mu_{\omega,t}^\kappa(X_t/t\in\cdot))_{t\in\N}$. Here the existence of an LDP with some rate function $J'$ follows from the argument of \cite[Theorem 1.1]{carmona_hu}. A concentration inequality corresponding to \cite[Proposition 2.3]{carmona_hu} has been proven for $\sigma=0$, $\rho=\delta_{\{-1\}}$ in \cite[Proposition 5.3]{thesis}. A similar proof applies to the general case, and can also be used to obtain the display following (3.2) in \cite{carmona_hu}. To identify $J'$ with $J^\kappa$ one can follow the arguments from \cite[Theorem 9.1]{comets_st_flour}. The convexity of $x\mapsto J^\kappa(x)$ can be shown in the same way as in \cite[Proposition 5.6]{thesis}. Finally, to show that $J^\kappa$ is good, it is enough to observe $J^\kappa(x)\geq I^\kappa(x)-(\a\vee 0)-\p(\kappa)\to\infty$ for $x\to\infty$.

Next, we claim that the LDP also holds for general $t\in\R_+$. Here one has to be somewhat careful, because the point-to-point partition function does not necessarily have an almost sure limit. That is, if $\rho(\{-1\})>0$ then it is proved in \cite[Proposition 5.6]{thesis}(i) that almost surely,
\begin{align}
\label{eq:non_existence}
-\infty=\liminf_{t\to\infty}\frac 1t\log Z_{\omega,t,[tx]}^{\kappa\mathbf1}<\lim_{t\to\infty,t\in\N}\frac 1t\log Z_{\omega,t,[tx]}^{\kappa\mathbf1}=\p(\kappa,x).
\end{align}
This is why the restriction to $t\in\N$ appears in \eqref{eq:p2p}. It in particular follows that 
\begin{align*}
 \liminf_{t\to\infty}\frac 1t\log \mu_{\omega,t}(X_t=[tx])=-\infty.
\end{align*}
But of course this does not contradict the LDP since we need a lower bound only for open sets. On the technical level, the divergence in~\eqref{eq:non_existence} is caused by the situation that $(t,[tx])$ is close behind a hard obstacle. Such a situation happens for a single point but should not happen for many points simultaneously. We show below that two points are enough to avoid this singularity. 

Now we start with the large deviation lower bound, 
\begin{align}
\label{eq:LDP_lower}
\liminf_{t\to\infty}\frac 1t\log \mu_{\omega,t}(X_t/t\in G)\geq -\inf_{x\in G}J^\kappa(x)\quad\text{ for all }G\subseteq\R^d\text{ open.}
\end{align}
Let $x\in G$ and let $t$ be large enough that both $\floor {tx}\in\floor{t}G$ and $\floor tx+e_1\in \floor tG$. Then
\begin{align*}
\mu_{t,\omega}(X_t/t\in G)&\geq (Z_{\omega,t}^\kappa)^{-1}E^\kappa\left[e^{H_t(\omega,X)}\1_{\{X_{\floor t}=\floor{tx},X_s\in \{\floor{tx},\floor{tx}+e_1\}\text{ for all }s\in[\floor t,t]\}}\right]\\
&=:\frac{Z_{\omega,\floor t,\floor{tx}}^\kappa}{Z_{\omega,t}^\kappa}A(\omega,t,\floor{tx}).
\end{align*}
By \eqref{eq:p2p}, almost surely,
\begin{align}\label{eq:A}
\liminf_{t\to\infty}\frac 1t\log \mu_{t,\omega}(X_t/t\in G)\geq -J^\kappa(x)+\liminf_{t\to\infty} \frac 1t\log A(\omega,t,\floor{tx}).
\end{align} 
Note that the law of $A(t,\omega,x)$ does not depend on $x$ or on the integer part of $t$. In the case $\sigma^2=0$ and $\rho=\delta_{\{-1\}}$, it is shown in~\cite[Lemma 6.4~(iii)]{thesis} that there exists $\delta>0$ such that
\begin{align*}
\sup_{t\in[0,1]} \E\left[ A(t,\omega,0)^{-\delta}\right]<\infty,
\end{align*}
and the general case can be proved in the same way. The Borel-Cantelli lemma then shows that the second term in \eqref{eq:A} equals zero almost surely, and since $x\in G$ is arbitrary we obtain \eqref{eq:LDP_lower}. For the large deviation upper bound, let $\emptyset\neq F\subseteq\R^d$ be closed and let $\overline{F_\eps}$ denote the closure of the $\eps$-neighborhood of $F$. By the LDP for integer times,
\begin{align}
\limsup_{t\to\infty}\frac 1t\log \mu_{\omega,t}(X_t\in tF)&=\limsup_{t\to\infty}\frac 1t\log \mu_{\omega,t}(X_t\in tF,X_{\floor t}\in\floor t\overline{F_\eps})\notag\\
&\leq \limsup_{t\to\infty}\frac 1t\log \mu_{\omega,t}(X_{\floor t}\in\floor t\overline{F_\eps})\notag\\
&\leq -\inf_{x\in \overline{F_\eps}}J^\kappa(x)+\limsup_{t\to\infty}\frac 1t\log \widetilde A(\omega,t),\label{eq:AA}
\end{align}
where 
\begin{align*}
\widetilde A(\omega,t):=\sum_{x\in\floor t\overline{F_\eps}}\frac{\mu_{\omega,\floor t}(X_{\floor t}=x)}{\mu_{\omega,\floor t}(X_{\floor t}\in \floor t\overline{F_\eps})}Z_{\theta^{\floor t,x}\omega,t-\floor t}^\kappa
\end{align*}
and $\theta^{t,x}\omega$ denotes the space-time shift of the environment. It is easy to see that $\log \widetilde A(t,\omega)$ has finite exponential moments, uniformly in $t$, so that the second term in \eqref{eq:AA} is zero almost surely by the Borel-Cantelli lemma. Since $J^\kappa$ is continuous, the first term converges to $-\inf_{x\in F}J^\kappa(x)$ for $\eps\downarrow 0$ as desired.
\end{proof} 

\section*{Acknowledgement}
RF is supported by ISHIZUE 2019 of Kyoto University Research Development Program. SJ is supported by a JSPS Postdoctoral Fellowship for Research in Japan, Grant-in-Aid for JSPS Fellows 19F19814.


\begin{thebibliography}{10}

\bibitem{ACM92}
Ho~S. Ahn, Ren\'{e}~A. Carmona, and Stanislas~A. Molchanov.
\newblock Nonstationary {A}nderson model with {L}\'{e}vy potential.
\newblock In {\em Stochastic partial differential equations and their
  applications ({C}harlotte, {NC}, 1991)}, volume 176 of {\em Lect. Notes
  Control Inf. Sci.}, pages 1--11. Springer, Berlin, 1992.

\bibitem{AD13}
Antonio Auffinger and Michael Damron.
\newblock The scaling relation {$\chi=2\xi-1$} for directed polymers in a
  random environment.
\newblock {\em ALEA Lat. Am. J. Probab. Math. Stat.}, 10(2):857--880, 2013.

\bibitem{BMRS19}
Rodrigo Bazaes, Chiranjib Mukherjee, Alejandro Ram\'{i}rez, and Santiago
  Saglietti.
\newblock Equality and difference of quenched and averaged large deviation rate
  functions for random walks in random environments without ballisticity.
\newblock \emph{preprint, arXiv:1906.05328}, 2019.

\bibitem{birkner_thesis}
Matthias Birkner.
\newblock {\em Branching particle systems; Genealogical construction; Long-time
  behaviour; Random environment; State dependent branching rate}.
\newblock Dissertation, Johann Wolfgang Goethe-Universit\"{a}t, Frankfurt,
  2003.
\newblock
  \url{http://publikationen.ub.uni-frankfurt.de/frontdoor/index/index/docId/5291}.

\bibitem{birkner}
Matthias Birkner.
\newblock A condition for weak disorder for directed polymers in random
  environment.
\newblock {\em Electron. Comm. Probab.}, 9:22--25, 2004.

\bibitem{carmona_hu}
Philippe Carmona and Yueyun Hu.
\newblock Fluctuation exponents and large deviations for directed polymers in a
  random environment.
\newblock {\em Stochastic Process. Appl.}, 112(2):285--308, 2004.

\bibitem{comets_st_flour}
Francis Comets.
\newblock {\em Directed polymers in random environments}, volume 2175 of {\em
  Lecture Notes in Mathematics}.
\newblock Springer, Cham, 2017.
\newblock Lecture notes from the 46th Probability Summer School held in
  Saint-Flour, 2016.

\bibitem{cometsyoshida}
Francis Comets and Nobuo Yoshida.
\newblock Brownian directed polymers in random environment.
\newblock {\em Comm. Math. Phys.}, 254(2):257--287, 2005.

\bibitem{yoshida_diffusive}
Francis Comets and Nobuo Yoshida.
\newblock {D}irected polymers in random environment are diffusive at weak
  disorder.
\newblock {\em Ann. Probab.}, 34(5):1746--1770, 2006.

\bibitem{shiga3}
Michael Cranston, Thomas.~S. Mountford, and T.~Shiga.
\newblock Lyapunov exponent for the parabolic {A}nderson model with {L}\'{e}vy
  noise.
\newblock {\em Probab. Theory Related Fields}, 132(3):321--355, 2005.

\bibitem{DZ}
Amir Dembo and Ofer Zeitouni.
\newblock {\em Large deviations techniques and applications}, volume~38 of {\em
  Applications of Mathematics (New York)}.
\newblock Springer-Verlag, New York, second edition, 1998.

\bibitem{shiga2}
Tasuku Furuoya and Tokuzo Shiga.
\newblock Sample {L}yapunov exponent for a class of linear {M}arkovian systems
  over {$\mathbf{Z}^d$}.
\newblock {\em Osaka J. Math.}, 35(1):35--72, 1998.

\bibitem{thesis}
Stefan Junk.
\newblock {\em Random polymers in disastrous environments}.
\newblock Dissertation, Technische Universit\"{a}t M\"{u}nchen, M\"{u}nchen,
  2019.
\newblock \url{http://mediatum.ub.tum.de/?id=1488489}.

\bibitem{self_comp}
Stefan Junk.
\newblock Comparison of partition functions in a space-time random environment.
\newblock {\em J. Stat. Phys.}, 181(1):95--115, 2020.

\bibitem{LacoinVSD}
Hubert Lacoin.
\newblock New bounds for the free energy of directed polymers in dimension
  {$1+1$} and {$1+2$}.
\newblock {\em Comm. Math. Phys.}, 294(2):471--503, 2010.

\bibitem{mountford}
Thomas~S. Mountford.
\newblock A note on limiting behaviour of disastrous environment exponents.
\newblock {\em Electron. J. Probab.}, 6:no. 1, 10 pp, 2001.

\bibitem{Nob97}
John~M. Noble.
\newblock The direct polymer in a random environment.
\newblock {\em Stochastic Anal. Appl.}, 15(4):585--612, 1997.

\bibitem{RAS14}
Firas Rassoul-Agha and Timo Sepp\"{a}l\"{a}inen.
\newblock Quenched point-to-point free energy for random walks in random
  potentials.
\newblock {\em Probab. Theory Related Fields}, 158(3-4):711--750, 2014.

\bibitem{RASY17b}
Firas Rassoul-Agha, Timo Sepp\"{a}l\"{a}inen, and Atilla Yilmaz.
\newblock Averaged vs. quenched large deviations and entropy for random walk in
  a dynamic random environment.
\newblock {\em Electron. J. Probab.}, 22:Paper No. 57, 47 pp, 2017.

\bibitem{seppalainen}
Timo Sepp\"{a}l\"{a}inen.
\newblock Scaling for a one-dimensional directed polymer with boundary
  conditions.
\newblock {\em Ann. Probab.}, 40(1):19--73, 2012.

\bibitem{shiga}
Tokuzo Shiga.
\newblock Exponential decay rate of survival probability in a disastrous random
  environment.
\newblock {\em Probab. Theory Related Fields}, 108(3):417--439, 1997.

\bibitem{Yil09}
Atilla Yilmaz.
\newblock Large deviations for random walk in a space-time product environment.
\newblock {\em Ann. Probab.}, 37(1):189--205, 2009.

\bibitem{Yil11}
Atilla Yilmaz.
\newblock Equality of averaged and quenched large deviations for random walks
  in random environments in dimensions four and higher.
\newblock {\em Probab. Theory Related Fields}, 149(3-4):463--491, 2011.

\bibitem{YZ10}
Atilla Yilmaz and Ofer Zeitouni.
\newblock Differing averaged and quenched large deviations for random walks in
  random environments in dimensions two and three.
\newblock {\em Comm. Math. Phys.}, 300(1):243--271, 2010.

\end{thebibliography}
\end{document}